\documentclass[a4paper,11pt]{amsart}
\usepackage[all]{xy}
\usepackage{amscd}
\usepackage{amsfonts}
\usepackage{a4wide}
\usepackage{amsmath, amsthm}
\usepackage{amssymb}
\usepackage{fullpage}
\usepackage[applemac]{inputenc}
\usepackage[pdftex,breaklinks,colorlinks,citecolor=blue]{hyperref}

\newtheorem{theorem}{Theorem}[section]

\newtheorem{lemma}[theorem]{Lemma}
\newtheorem{proposition}[theorem]{Proposition}

\newcommand{\F}{\mathbb F}

\title{The $R$-transform as power map and its generalisations to higher degree}

\author{Alp Bassa  $^\P$}
\address{Alp Bassa. Bo\u{g}azi\c{c}i University,
Faculty of Arts and Sciences,
Department of Mathematics,
34342 Bebek, \.{I}stanbul,
Turkey,
}
\email{alp.bassa@boun.edu.tr}
\author{Ricardo Menares $^\dag$}
\address{Ricardo Menares. Ricardo Menares. Pontificia Universidad Cat\'olica de Chile,  Facultad de Matem\'aticas, Vicu\~na Mackenna 4860, Santiago, Chile.}
\email{rmenares.v@gmail.com}
\thanks{$\P$ Bo\u gazi\c ci University}
\thanks{\dag Pontificia Universidad Cat\'olica de Chile}

\begin{document}

\begin{abstract}
We give iterative constructions for irreducible polynomials over $\mathbb F_q$ of degree $n\cdot t^r$ for all $r\geq 0$, starting from irreducible polynomials of degree $n$. The iterative constructions correspond modulo fractional linear transformations to compositions with power functions $x^t$. The $R$-transform introduced by Cohen is recovered as a particular case corresponding to $x^2$, hence we obtain a generalization of Cohen's $R$-transform ($t=2$) to arbitrary degrees $t\geq 2$. Important properties like self-reciprocity and invariance of roots under certain automorphisms are deduced from invariance under multiplication by appropriate roots of unity. Extending to quadratic extensions of $\mathbb F_q$ we recover and generalize a  recursive construction of Panario, Reis and Wang.
\end{abstract} 
\maketitle
\tableofcontents

\section{Introduction}
The primary way of constructing irreducible polynomials of high degree over a finite field $\mathbb F_q$ is to start with an irreducible polynomial $f$ and to repeatedly apply a given transformation, which generally amounts to composition with a fixed polynomial or rational function. This way an infinite sequence of polynomials over $\mathbb F_q$ of increasing degree is obtained. Characterising the polynomials $f$, for which the members of the resulting sequence of polynomials are all irreducible and satisfy further desirable properties has become a highly non-trivial and important question. There are many works on the different aspects of this problem, including  \cite{AAK}, \cite{Agou}, \cite{Cohen82}, \cite{Cohen92}, \cite{recurrent}, \cite{recurrent2}, \cite{Varshamov}. 

When $q$ is odd, one of the most important transformations for the construction of irreducible polynomials is the $R$-transform introduced by Cohen (\cite{Cohen92}), which corresponds to composition with the rational function
$$\frac{1}{2}\left(x+\frac{1}{x}\right),$$ 
attaching to a polynomial $g(x)\in \mathbb F_q[x]$ the polynomial
$$g^R(x)=(2x)^{\deg g}\cdot g\left(\frac{1}{2}\left(x+\frac{1}{x}\right)\right).$$
Cohen gives the following characterisation of polynomials $g$ for which the members of the resulting sequence are irreducible:
\begin{theorem}[\cite{Cohen92}] \label{theorem_cohen}
Suppose $q$ is odd and let $g(x) \in \mathbb F_q[x]$ be a monic irreducible polynomial. Assume that $g(-1)\cdot g(1)$ is not a square in $\mathbb F_q$.  If $q \equiv 3 \mod 4$, assume moreover that $\deg g$ is even. Consider the sequence of polynomials $(g_m(x))_{m\geq 0}$ in $\mathbb F_q[x]$ defined by $g_0(x)=g(x)$ and  $$g_m(x)=g^R_{m-1}(x), \quad m \geq 1.$$
Then, $g_m(x)$ is an irreducible polynomial of degree $2^m\deg g$ for all $m$. 
\end{theorem}

This construction is particularly important from the complexity point of view, since the resulting polynomials are self-reciprocal. Moreover, for $q\equiv 1 \mod 4$, starting with a linear polynomial satisfying the conditions of Theorem~\ref{theorem_cohen}, we obtain a sequence of normal (see \cite{meyn}), and in fact even completely normal (see \cite{chapman}) polynomials over $\mathbb F_q$. 

In this paper, we start from the observation that the $R$-transform, after an appropriate change of coordinates by a fractional linear transformation, corresponds to composition with the square map $x^2$. This fact had been previously obtained by Ugolini (\cite{U14}, \cite{U15}), see also  \cite{chapman}. Then, we replace the square map by more general power maps  $x^t$, for some  fixed integer $t\geq 2$, and take the rational function to be an appropriate change of coordinates by a fractional linear transformation. This approach leads to simple conditions ensuring the irreducibility of the resulting sequence and symmetry properties akin to being self-reciprocal (cf. Theorem \ref{mainthm}, Theorem \ref{ThmMcNay} and Theorem \ref{ThmPanarioetal}).  Indeed, consider the sequence given by $g_0(x)=g(x)$ and $g_m(x)=g_{m-1}(x^t)$ for $m \geq 1$.  Using classical facts on Kummer extensions it is possible to identify very simple hypothesis on $g(x)$ and $t$ ensuring that all polynomials in the resulting sequence are irreducible, see Section \ref{power} for precise statements and references. Also, the set of  roots of $g_m(x)=g(x^{t^m})$ is invariant under multiplication by $t^m$-roots of unity and this feature leads to symmetry properties of the polynomials.

In order to describe the new transforms that arise, we introduce some notations. Consider the group ${\rm GL}_2(\mathbb F_q)$ of $2 \times 2$ matrices 
$$\sigma = \left[ {\begin{array}{cc}
   a & b \\       c & d \\      \end{array} } \right]  \in M_2(\mathbb F_q)$$
   such that $ad-bc \neq 0$. This group acts on $\hat{\mathbb F}_q := \overline{\mathbb F}_q\cup \{\infty\}$ by M\"obius transformations, as
    $$\alpha \in \overline{\mathbb{F}}_q, \quad \sigma \cdot \alpha =\left\{ \begin{array}{cl}
  \frac{a\alpha +b}{c\alpha+d} & \textrm{ if } c\alpha + d \neq 0\\
  \infty & \textrm{ otherwise }
  \end{array}\right., \qquad \sigma \cdot \infty =\left\{ \begin{array}{cl}
  \frac{a}{c} & \textrm{ if } c\neq 0\\
  \infty & \textrm{ otherwise. }
  \end{array}\right.  $$
  
For any degree $n$ polynomial $f(x)$, we set $f(\infty):=\infty$ and denote by $P_\sigma(f)(x)$ the polynomial 

\begin{equation}\label{Psigma}
P_\sigma(f) (x) := (cx+d)^n f \left( \frac{ax+b}{cx+d}\right).
\end{equation}

The works \cite{Garefalakis}, \cite{ST} have shown fruitful, with regards to the construction of irreducible poylnomials, to consider the rule \eqref{Psigma} as an action of ${\rm GL}_2(\mathbb F_q)$ on appropriate classes of polynomials. In particular, much effort has been  brought to the characterization of polynomials which are invariant under subgroups of  ${\rm GL}_2(\mathbb F_q)$, e.g. see \cite{MP}, \cite{Panarioetal}, \cite{Reis}, \cite{Reis2} and \cite{S}.

 For any positive integer $t$, let $S_t :   \mathbb F_q[x] \rightarrow \mathbb F_q[x]$ be the linear map given by $S_t(f)(x):=f(x^t).$ Then, we define a transformation $R_{\sigma,t}$ as follows. For any nonzero polynomial $g(x) \in \mathbb F_q[x]$ of degree $n$,  we define the polynomial $g^{R_{\sigma,t}}(x) \in \mathbb F_q[x]$ by $$ g^{R_{\sigma,t}}(x):= P_{\sigma^{-1}} \circ S_t \circ P_{\sigma}(g)(x).$$

Whenever $g(a/c) \neq 0$, we define an element $\eta(g;\sigma) \in \mathbb F_q$ by 
$$\eta(g;\sigma):=(\sigma^{-1}\cdot \infty)^n \cdot \frac{g(\sigma\cdot 0)}{g(\sigma \cdot \infty)} =  \left(-\frac{d}{c}\right)^n \cdot g\left(\frac{b}{d}\right)\cdot g\left(\frac{a}{c}\right)^{-1},$$  with  the convention that  $\eta(g;\sigma)=(-d/a)^n \cdot g(b/d)$ for $c=0$ and $\eta(g;\sigma)=(-b/c)^n\cdot g\bigl(a/c\bigr)^{-1}$ for $d=0$.
   
Our main result is a  generalization of Theorem \ref{theorem_cohen} for the transforms $R_{\sigma,t}$. 
   
\begin{theorem} \label{mainthm}
Consider a finite field $\mathbb F_q$. Let $t\geq 2$ be an integer such that every prime factor of $t$ divides $q-1$. Let  $\sigma=\left[ {\begin{array}{cc}
   a & b \\       c & d \\      \end{array} } \right]\in {\rm GL}_2(\mathbb F_q)$.  Let $g(x)\neq x-a/c$ be a monic irreducible polynomial in $\mathbb F_q[x]$ of degree $n$. If $q\equiv 3 \mod 4$ and $t$ is even assume moreover that $n$ is even.

Assume that,  for all prime numbers $\ell|t$, the element $\eta(g;\sigma)$ is not an $\ell$-th power in $\mathbb F_q$. Define $g_0=g$ and let $g_m=g_{m-1}^{R_{\sigma,t}}$ for $m\geq 1$. Then $(g_m)_{m\geq 0}$ forms an infinite sequence of irreducible polynomials, with $\deg g_m=t^m\cdot n$.

Moreover, for $m\geq 0$ let $\zeta$ be a $t^m$-th root of unity. Then, the set of roots of $g_m$ is invariant under the action of the order $t^m$ matrix in ${\rm GL}_2\big(\mathbb{F}_q(\zeta)\big)$ given by
\begin{equation} \label{thematrix} 
M_{\sigma, \zeta}=\left[ {\begin{array}{cc}
\zeta ad-bc & (1-\zeta)ab \\ (\zeta-1)cd & ad-\zeta bc\\ \end{array} } \right].
\end{equation}

\end{theorem}
   
 In the  work \cite{Aravena}, sufficient conditions on the starting polynomial are established in order to ensure that the polynomials  constructed in the previous statement are completely  normal, thus generalizing  \cite[Theorem 1]{chapman} to this context.

Assume $q$ is odd. Set $\sigma^* :=  \left[ {\begin{array}{rr}
   1 & 1 \\      1 & -1 \\      \end{array} } \right] \in GL_2(\mathbb F_q) $. Then, whenever $g(1)\neq 0$, we have that $g^{R_{\sigma^*,2}} = g^R$. This shows that the $R$-transform, after a coordinate change by the fractional linear transformation $\frac{x+1}{x-1}$, corresponds to the composition with the power function $x^2$.  Moreover, taking $\zeta=-1$,  the action of $M_{\sigma^*,-1}$ is given by  $\alpha \mapsto \frac{1}{\alpha}$.  Thus, the stated invariance of the roots boils down to the polynomial being self-reciprocal, after being normalized to be monic. Also, we have that $\eta(\sigma^*;g)=g(1)\cdot g(-1)^{-1}$ differs by a square from $g(1)\cdot g(-1)$, thus showing that Theorem \ref{theorem_cohen} is implied by Theorem \ref{mainthm}.


With regards to other remarkable transforms that can be found in the literature, in order to realize them as a power map, sometimes it is necessary to perform a change of variables in a proper extension $\mathbb F_{q'}/\mathbb F_q$. More precisely, if $\sigma \in GL_2(\mathbb F_{q'})$ and $t \geq 2$ are such that for any polynomial $g(x) \in \mathbb F_q[x]$, a non zero constant multiple of the polynomial $g^{R_{\sigma,t}}(x)$ has coefficients in the base field $\mathbb F_q$, then Theorem \ref{mainthm} can be used to obtain results about the irreducibility over $\mathbb F_q$ of the iterates. In Section \ref{quadratic} we work out the details in the quadratic case $q'=q^2$ and apply this discussion to McNay's transform (see \cite{chapman}, section 3), see Theorem \ref{ThmMcNay} below.  We also apply these methods to a transform introduced by Panario, Reis and Wang in \cite{Panarioetal} which is related to Singer subgroups of $GL_2(\mathbb F_q)$. In particular,  Theorem \ref{ThmPanarioetal} below provides a partial answer to Problem 1 in \emph{loc. cit.}, Section 5.

Our proof of Theorem \ref{mainthm} begins by relating the given sequence to an appropriate sequence of composition with power maps. A technical issue comes from the fact that  $\deg  P_\sigma(f)=\deg f$ if and only if $f(\sigma\cdot \infty) \neq 0$ (see Proposition \ref{matac} i) below). For this reason, it is customary to consider the action of $GL_2(\mathbb F_q)$ on the space of polynomials that do not vanish on $\mathbb{F}_q$ in order to obtain a genuine group action. When dealing with nonlinear irreducible polynomials, as in most of the aforementioned works, this feature holds. In this work we do not want to rule out linear polynomials, as they can be useful at the start of the sequence, and hence we must take into account a possible defect on the rule \eqref{Psigma} to be a group action. This is of course a minor technical issue, that nevertheless forces us to do some extra bookkeeping (cf. Lemma \ref{pia}).

We observe that the irreducibility of $g_1$ implies the irreducibility of the whole sequence (see  Theorem \ref{general}). Then, we use irreducibility criteria classically known. In the same vein, the roots of $f(x^t)$ are invariant under multiplication by $t$-roots of unity. This is used to show that the set of roots of the polynomials obtained by iterating the transform $R_{\sigma,t}$ are invariant under the action of the cyclic subgroups of ${\rm GL}_2(\mathbb F_q)$ stated in Theorem \ref{mainthm}.

{\bf Acknowledgements}. Alp Bassa was partially supported by the BAGEP Award of the Science Academy with funding supplied by Mehve{\c s} Demiren in memory of Selim Demiren and by Bo\u gazi\c ci University Research Fund Grant Number 15B06SUP2 and by Conicyt-MEC 80130064 grant. Ricardo Menares was partially supported by  FONDECYT 1211858  grant.



\section{Composition with power functions}\label{power}
Fix a finite field $\mathbb F_q$.  For an irreducible polynomial $f(x)\in \mathbb F_q[x]$, it is precisely known, for which positive integers $t$, the polynomial $f(x^t)$ is irreducible. We have the following result
\begin{theorem}[\cite{Cohen69}, Theorem 1] \label{menezestheorem}
Let $f(x)\in \mathbb F_q[x]$ be an irreducible polynomial of degree $n$ and exponent $e$ (equal to the order of any root of $f(x)$). For a positive integer $t$ the polynomial $f(x^t)$ is irreducible over $\mathbb F_q$ if and only if
\begin{enumerate}
\item ${\rm gcd}(t,(q^n-1)/e)=1$,
\item each prime factor of $t$ divides $e$, and
\item if $4|t$ then $4|(q^n-1)$.
\end{enumerate}
\end{theorem}

We remark the following: suppose $t$ factors into primes as $p_1^{e_1}\cdot p_2^{e_2}\cdots p_r^{e_r}$. If $t$ satisfies the condition of the theorem, then so does $t'=p_1^{f_1}\cdot p_2^{f_2}\cdots p_r^{f_r}$ with $f_i\geq e_i$, unless $4\nmid t$ and $4 \mid t'$. So if $f(x^t)$ is irreducible, then increasing the exponents of primes in the factorisation of $t$ without changing the prime factors in the factorisation does not change the irreducibility of $f(x^t)$. As a particular case, for $t'=t^r$, we obtain 

\begin{proposition} \label{poweriterative} Let $f(x)\in \mathbb F_q[x]$ be a  polynomial of degree $n$ and fix a positive integer $t$. If $t$ is even assume moreover that $4|(q^n-1)$. Suppose $f(x^t)$ is irreducible. Then for all $r\geq 0$, the polynomial $f(x^{t^r})$ is irreducible.
\end{proposition}

We rephrase Proposition \ref{poweriterative} in a iterative manner:
\begin{proposition}\label{iteration}  Let $f(x)\in \mathbb F_q[x]$ be a  polynomial of degree $n$ and fix a positive integer $t$. If $t$ is even assume moreover that $4|(q^n-1)$. Let $f_0(x)=f(x)$ and for $i\geq 1$ define $f_i(x)=f_{i-1}(x^t)$. Assume that $f_1(x)$ is irreducible. Then, for all $i\geq 0$, the polynomial $f_i(x)$ is irreducible of degree $(\deg f)\cdot t^i$.
\end{proposition}

We will see below that the iterative construction of self-reciprocal irreducible polynomials using the $R$-transform introduced by Cohen corresponds (up to a change of variables) to the sequence above in the case where  $q$ is odd and $t=2$.

\section{Generalities on the action of $GL_2(\mathbb{F}_q)$ on polynomials}\label{action}
 We recall the following classical results on irreducible polynomials over finite fields:

\begin{proposition}\label{matac} Let $g(x)\in \mathbb F_q[x]$ be a polynomial of degree $n$. Let $a,b,c,d\in \mathbb F_q$ with $ad-bc\neq 0$, so that 
\begin{equation}\label{sigma}
\sigma =\left[ {\begin{array}{cc}
   a & b \\       c & d \\      \end{array} } \right] \in {\rm GL}_2(\mathbb{F}_q)
   \end{equation}
is an invertible matrix over $\mathbb F_q$. Consider the polynomial $P_\sigma(g)(x)$ as defined in \eqref{Psigma}. Then, 

\begin{enumerate}
\item[i)] $P_\sigma(g)(x)$ is of degree $n$ if and only if $g(\sigma\cdot \infty)\neq 0$
\item[ii)] if $g(\sigma\cdot \infty)\neq 0$,  then  for every $\tau \in GL_2(\mathbb F_q)$, we have that 
$$ P_\tau \circ P_\sigma (g) = P_{\sigma\cdot\tau  }(g)$$

\item[iii)] if $g(x)$ is irreducible,  then $P_\sigma(g)(x)$  is irreducible

\item[iv)] if $P_\sigma(g)(x)$  is irreducible and of degree $n$, then $g(x)$ is irreducible 
\end{enumerate}
\end{proposition}

\begin{proof}
Part iii) is Corollary 3.8 in  \cite{menezes}. Let $a_n\neq 0$ be the leading coefficient of $g(x)$. Then, the coefficient of $x^n$ in $P_\sigma(g)(x)$  is $a_n\cdot a^n$ if $c=0$ and $c^n\cdot g(a/c)$ otherwise. Since $ad-bc \neq 0$, we have that $c=0$ implies $a \neq 0$. These relations prove part i).  Part ii) is a consequence of part i) and a direct calculation (cf. Lemma 2.2 in \cite{ST}). The hypothesis in  part iv), together with part i) imply that $g(\sigma\cdot \infty)\neq 0$. Hence we can use part ii) with $\tau=\sigma^{-1}$ to conclude that $g=P_{\sigma^{-1}}\circ P_\sigma(g)$. We see that $g$ is irreducible using part iii).  
\end{proof}


The following statement is a mild generalization of Lemma 1 in  \cite{Garefalakis}. 

\begin{lemma}\label{roots}
Let $g(x) \in \mathbb F_q[x]$  be a  polynomial of degree $n$ and  let $f=P_\sigma(g)$. Then, $f(\sigma^{-1}\cdot \infty) \neq 0$. Moreover, if  $\alpha_1,\alpha_2, \ldots \alpha_m$, with $m\leq n$, are the roots of $f$, then the elements  $\beta_1, \beta_2, \ldots \beta_m$ given by
\begin{equation}\label{rootsrelation}
\beta_i={\sigma}\cdot \alpha_i=\frac{a\alpha_i+b}{c\alpha_i+d}, \quad1\leq i\leq m,
\end{equation}
are  roots of $g$. 
\end{lemma}

\begin{proof} 
Since $f(\infty)=\infty$, we may assume that $\sigma^{-1}\cdot \infty \neq \infty$, which amounts to the condition $c \neq 0$. Writing $g(x)=\sum_{i=0}^na_ix^i$, we see that 

$$f(x)=(cx+d)^ng\left(\frac{ax+b}{cx+d}\right)=\sum_{i=0}^na_i(ax+b)^i(cx+d)^{n-i}.$$

Hence, $f(\sigma^{-1}\cdot \infty) = f(-d/c)=a_n\left(-\det \sigma/c\right)^n \neq 0,$ as claimed. Also from the formula above we see that $f(\alpha)=0$ if and only if $c\alpha+d\neq 0$ and $g(\frac{a\alpha+b}{c\alpha+d})=0$, as desired. 
\end{proof}


The following Lemma will be applied in the next section when the starting polynomial is linear, and hence has a root in $\mathbb{F}_q$. 

\begin{lemma}\label{pia}
Let $g_0(x) \in \mathbb F_q[x]$ be a polynomial of degree one and let $t\geq 2$ be a positive integer. Let $\sigma \in {\rm GL}_2(\mathbb F_q)$ as in \eqref{sigma} and define $g_1:=g_{0}^{R_{\sigma,t}}$.  Assume that
\begin{enumerate}
\item[i)] $g_0(\sigma \cdot \infty) \neq 0$ 
\item[ii)] $g_1$ is irreducible. 
\end{enumerate}

Let $f_0:=P_\sigma g_0$. Then, $f_0\big( (\sigma^{-1}\cdot \infty)^t\big) \neq 0$.

\end{lemma}

\begin{proof}

Assume for contradiction that $f_0\big( (\sigma^{-1}\cdot \infty)^t\big) = 0$. Write $g_0(x)=B(x-\alpha)$ with  $B, \alpha \in \mathbb{F}_q, B \neq 0$ and $\alpha \neq \sigma\cdot \infty$. Then, Lemma \ref{roots} implies that $\alpha=\sigma\cdot (\sigma^{-1}\cdot \infty)^t$. Hence, the condition $\alpha \neq \sigma \cdot \infty$ amounts to $c \neq 0$.
 
We have that $$g_1(x) = P_{\sigma^{-1}}\circ S_t(f_0)(x)= \left( \frac{-cx+a}{\det \sigma} \right)^t f_0\bigg( \Big( \frac{dx-b}{-cx+a}\Big)^t \bigg).$$
 We observe that $p \nmid t$, since otherwise the previous expression and the surjectivity of the Frobenius map would show that $g_1$ is the $p$-th power of a polynomial, hence not irreducible. On the other hand, from the formula

$$g_1(x) = B(\det\sigma)^{-t} \bigg( a(dx-b)^t+b(-cx+a)^t -\alpha \big( c(dx-b)^t +d(-cx+a)^t\big) \bigg),$$

\noindent we infer $g_1(\sigma \cdot \infty)=\frac{B}{c^t}\cdot(a-\alpha c).$
This expression cannot vanish since the relation $a-\alpha c=0$ is equivalent to $\alpha=\sigma \cdot \infty$, which is excluded by hypothesis.

Since  $g_1(\sigma \cdot \infty)\neq 0$, we deduce that $g_1(u)=0$ if and only if $f_0\left( (\sigma^{-1} \cdot u)^t\right)=0$. Hence, $u$ is a root of $g_1$ if and only if 
$(\sigma^{-1} \cdot u)^t=(\sigma^{-1}\cdot \infty)^t$. We conclude that the set of roots of $g_1$ is $$S=\{ \sigma \cdot \big( \zeta (\sigma^{-1}\cdot \infty)\big) : \zeta^t=1\}. $$

We remark that the relation $\sigma^{-1} \cdot \infty = 0$ implies $\alpha=\sigma \cdot 0 = \infty$, which is not possible. Hence, $\sigma^{-1} \cdot \infty \neq 0$. On the other hand, $\sigma^{-1}\cdot \infty \neq \infty$ because $c \neq 0$. We conclude that the set $S$ has as many elements as the set of $t$-th roots of unity in $\overline{\mathbb{F}}_q$. Since $p \nmid t$, this implies that $S$ has $t$ elements. Hence, $\deg g_1=t=\deg S_t(f_0)$. Then, Proposition \ref{matac} i) ensures that
  $f_0 \big( (\sigma^{-1}\cdot \infty)^t\big) = S_t(f_0)(\sigma^{-1}\cdot \infty)  \neq 0$, as claimed. 
\end{proof}

\section{Proof of the main theorem}

We first show that the irreducibility of $g_1$ implies the irreducibility of the whole sequence. Then we prove Theorem \ref{mainthm} at the end of the section.

\begin{theorem} \label{general}
Let $g_0(x)\in \mathbb F_q[x]$ be an irreducible polynomial of degree $n$ and let $t\geq 2$ be a positive integer. If $q\equiv 3 \mod 4$ and $t$ is even assume moreover that $n$ is even. Let $\sigma \in {\rm GL}_2(\mathbb F_q)$ as in \eqref{sigma}. For $m\geq 1$ define $g_m=g_{m-1}^{R_{\sigma,t}}$.  Assume that
\begin{enumerate}
\item[i)] $g_0(\sigma \cdot \infty) \neq 0$ 
\item[ii)] $g_1$ is irreducible. 
\end{enumerate}

Then, $(g_m)_{m\geq 0}$ is a sequence of irreducible polynomials over $\mathbb F_q$ with $\deg g_m=t^mn$.

For $m\geq 0$, let $\zeta$ be a primitive $t^m$-th root of unity. Then the roots of $g_m$ are invariant under the action of the order $t^m$ matrix $M_{\sigma, \zeta}$ given by \eqref{thematrix}.
\end{theorem}


\begin{proof} Define $f_m := P_\sigma (g_m)$ for $m \geq0$. Since $g_0$ and $g_1$ are irreducible and  $g_0 (\sigma \cdot \infty) \neq 0$, Proposition \ref{matac} i) and iii) ensure that $f_0$ and $f_1$ are also irreducible and $\deg f_0 = n$. 

We now show that $f_m(x)=f_{m-1}(x^t)$ for all $m \geq 1$. We proceed by induction on $m$. We have that  $S_t \circ f_0(\sigma^{-1}\cdot \infty) = f_0\big( (\sigma^{-1}\cdot \infty)^t\big) \neq 0$ due to the irreducibility of the non linear polynomial $f_0$ if $n\geq2$ and due to Lemma \ref{pia} if $n=1$.  Hence, Proposition \ref{matac} ii) ensures that $$f_1(x) = P_{\sigma}\circ P_{\sigma^{-1}}\circ S_t \circ f_0 (x)=S_t\circ f_0(x) = f_0(x^t),$$ thus proving the claim for $m=1$. Now assume the claim holds for $m\geq 1$.  Since  $$S_t \circ f_m (\sigma^{-1}\cdot \infty) = f_m\big( (\sigma^{-1}\cdot \infty)^t\big) = f_1\big( (\sigma^{-1}\cdot \infty)^{t^{m-1}}\big)$$ cannot vanish due to the irreducibility of the non linear polynomial $f_1$, using Proposition \ref{matac} ii) we have that $$f_{m+1}(x) = P_{\sigma}\circ P_{\sigma^{-1}} \circ S_t \circ f_m(x)=f_m(x^t),$$ as claimed. Then, since $f_0$ and $f_1$ are irreducible, we deduce from Proposition \ref{iteration} that $f_m$ is irreducible for all $m \geq 1$. 

On the other hand, we  have that $\deg g_m = \deg f_m$. Indeed, since $f_m=P_\sigma g_m$, we clearly have that $\deg f_m\leq \deg g_m$ and  from $\deg g_m =\deg P_{\sigma^{-1}} \circ S_t \circ P_\sigma g_{m-1}\leq t\deg g_{m-1}$ we deduce the reverse inequality $\deg g_m \leq t^m\cdot n = \deg f_m$.   Hence, applying Proposition \ref{matac} iv) we deduce that $g_m$ is irreducible.

Let $\zeta$ be a $t^m$-th root of unity.  Note that since $f_m(x)=f_{0}(x^{t^m})$, the roots of $f_m$ are invariant under the multiplication by $\zeta$ map $\mu_{\zeta}:\beta \mapsto \zeta\beta$. Using Equation~\eqref{rootsrelation}, we see that the roots of $g_m$ are invariant under the map 
$$\sigma^{-1}\circ\mu_{\zeta}\circ \sigma: \alpha\mapsto \frac{(\zeta ad-bc)\cdot \alpha+(\zeta-1)bd}{-(\zeta-1)ac\cdot \alpha+ad-\zeta bc}.$$  This map is of order $t^m$ and can be described using the matrix 
\begin{equation*} 
\left[ {\begin{array}{cc}
\zeta ad-bc & (\zeta-1)bd \\ -(\zeta-1)ac & ad-\zeta bc\\ \end{array} } \right] 
\in  {\rm GL}_2(\mathbb F_q(\zeta)).
\end{equation*}

 \end{proof}

{\bf Proof of Theorem \ref{mainthm}:} In view of Theorem \ref{general}, we only need to check that $g_1$ is irreducible under the stated assumptions. 



 Let $f_0:=P_\sigma(g_0), f_1:=P_\sigma(g_1)$. Since $g_0(\sigma\cdot \infty) \neq 0$, we have that $f_0$ is irreducible of degree $n$ by Proposition \ref{matac} i) and iii). 
 
 {\bf Claim}. We have that $f_0\left((\sigma^{-1}\cdot \infty)^t\right) \neq 0.$ 
 
 Assume for contradiction that $f_0\left((\sigma^{-1}\cdot \infty)^t\right)= 0.$ Since $f_0$ is irreducible, we have that $n=1$. Then, $g_0$ is of the form $g_0(x)=B\cdot (x-\alpha)$ with  $B, \alpha \in \mathbb{F}_q, B \neq 0$ and $\alpha \neq \sigma\cdot \infty$. Moreover, Lemma \ref{roots} implies that $\alpha= \sigma \cdot (\sigma^{-1}\cdot \infty)^t$. A short calculation shows that in this case $\eta(g;\sigma)=(\sigma\cdot \infty)^t$ is a $t$-th power, contradicting the hypothesis. 
 
Using the Claim and Proposition \ref{matac} ii), we have that $$f_1(x)=P_\sigma   \circ P_{\sigma^{-1}}\circ S_t(f_0)(x) =  f_0(x^t).$$
 
Now we show that $f_1(x)$ is irreducible, by checking the hypothesis of Theorem~\ref{menezestheorem}. Let $\beta$ be a root of $f_0$ and let $e$ be the order of $\beta$ in $\mathbb F_{q^n}^*$. Assumption (3) is clearly implied by our hypothesis on $t$. If we assume that assumption (1) holds, then assumption (2) holds as well. Indeed, let $\ell |t$ be a prime factor. Since $\ell | (q-1)$ and (1) implies that $\ell$ does not divide $(q^n-1)/e$, we conclude that $\ell | e$. 
 
 We are thus reduced to  check hypothesis (1) in Theorem \ref{menezestheorem}. Assume for contradiction that there is a prime number $\ell$ such that $\ell |\gcd(t,(q^n-1)/e)$. Then, $(q^n-1)/(e\ell)$ is an integer, and we have that 
 $$\beta^{\frac{q^n-1}{\ell}}=(\beta^e)^{\frac{q^n-1}{e\ell}}=1.$$ 

On the other hand, using that $\ell | q-1$, we have that 

$$\beta^{\frac{q^n-1}{\ell}}=\bigl(\beta^{q^{n-1}+q^{n-2}+\ldots+q+1}\bigr)^{(q-1)/\ell}=\bigl({\rm N}_{\mathbb F_{q^n}/\mathbb F_q}(\beta)\bigr)^{(q-1)/\ell},$$
where ${\rm N}_{\mathbb F_{q^n}/\mathbb F_q}$ denotes the norm map corresponding to the extension $\mathbb F_{q^n}/\mathbb F_q$. We conclude that ${\rm N}_{\mathbb F_{q^n}/\mathbb F_q}(\beta)$ is an $\ell$-th power in $\mathbb F_q$. 

We express this condition in terms of the coefficients of $g_0$. Write $$f_0(x)=b_nx^n+b_{n-1}x^{n-1}+\cdots+b_0, \quad g_0(x)=x^n+a_{n-1}x^{n-1}+\cdots+ a_1x+a_0.$$ Then, ${\rm N}_{\mathbb F_{q^n}/\mathbb F_q}(\beta)=(-1)^nb_0/b_n$ and we have that
$$f_0(x)=P_\sigma(g_0)(x)=(cx+d)^{n}g_0\left(\frac{ax+b}{cx+d}\right).$$
So 
$$b_0=d^{n}g_0\left(\frac{b}{d}\right) \text{ and } b_n=c^{n}g_0\left(\frac{a}{c}\right),$$
with the convention that $b_0=b^n$ if $d=0$  and $b_n=a^n$ if  $c=0$. Then, ${\rm N}_{\mathbb F_{q^n}/\mathbb F_q}(\beta)=(-1)^nb_0/b_n=\eta(g_0;\sigma)$. This is a contradiction, as $\eta(g_0;\sigma)$  is not an $\ell$-th power in $\mathbb F_q$. We conclude that $f_1$ is irreducible. 

Since $\deg f_1 \leq \deg g_1$ and $\deg g_1 \leq t \deg g_0 = nt=\deg f_1$, we have that $\deg f_1=\deg g_1$. Hence, $g_1$ is irreducible by Proposition \ref{matac} iv). $\square$


\section{Coordinate changes over a quadratic extension}\label{quadratic}

Let $\sigma \in {\rm GL}_2(\mathbb F_{q^2})$ and $t \geq 2$. We say that $R_{\sigma,t}$ is defined over $\mathbb F_q$ if   for any polynomial $g(x) \in \mathbb F_q[x]$ with $g(\sigma \cdot \infty) \neq 0$,  there exists an element $\kappa \in \mathbb F_{q^2}^*$ such that  $\kappa \cdot g^{R_{\sigma,t}}(x) \in \mathbb F_q[x]$. 

Let $g(x) \in \mathbb F_q[x]$ be an irreducible polynomial of degree $n$. If $n$ is odd, then $g(x)$ is irreducible in $\mathbb F_{q^2}[x]$. If $n$ is even, then it has exactly two irreducible factors $r(x),s(x) \in \mathbb F_{q^2}[x]$, both of degree $n/2$. 

Assume $R_{\sigma,t}$ is defined over $\mathbb F_q$. Consider the sequence $g_0=g$ and $g_m=g_{m-1}^{R_{\sigma,t}}$ for $m\geq 1$. If $n$ is even, define in a similar fashion sequences $r_m$ and $s_m$ starting from $r$ and $s$. Clearly, $g_m=r_m\cdot s_m$ in this case.

\begin{theorem}\label{quadraticit}
 Let $\sigma \in {\rm GL}_2(\mathbb F_{q^2})$ and $t \geq 2$ be such that $R_{\sigma,t}$ is defined over  $\mathbb F_q$. Assume that every prime factor of $t$ divides $q^2-1$. Let $g(x) \in \mathbb F_q[x]$ be an irreducible polynomial of degree $n$ such that $g(\sigma \cdot \infty) \neq 0$. 
 
Consider the sequence $g_m$ defined as above as well as $r_m$ if $n$ is even. If $n$ is odd (resp. even), assume that for all prime numbers $\ell |t$, the element $\eta(g;\sigma)$ (resp. $\eta(r;\sigma)$) is not an $\ell$-th power in $\mathbb F_{q^2}$.

Then, if $n$ is odd or if $4|nt$, we have that a nonzero multiple of $g_m(x)$ is an irreducible polynomial in $\mathbb F_q[x]$ for all $m \geq 0$. If $n$ is even, a nonzero multiple of $r_m$  is an irreducible polynomial in $\mathbb F_{q^2}[x]$ for all $m \geq 0$. 

Moreover, if $n$ is odd (resp. if $n$ is even), the roots of $g_m$ (resp. $r_m$)  are invariant under the action of the matrix $M_{\zeta,\sigma}$ given by \eqref{thematrix} (here, $\zeta$ is a primitive $t^m$-root of unity).
\end{theorem}

\begin{proof}
Since $q^2 \neq 3 \mod 4$, we can apply Theorem \ref{mainthm} with $q^2$ in place of $q$ to conclude that if $n$ is odd, all polynomials $g_m(x)$ are irreducible in $\mathbb F_{q^2}[x]$. Then, taking appropriate nonzero multiples, they are irreducible in $\F_q[x]$ as well. 

If $n$ is even, the same argument shows that  $r_m$ is irreducible in $\mathbb F_{q^2}[x]$ for all $m \geq 0$. Moreover,   $\deg(r_m)=(n/2)t^m$.

Assume $4|nt$. If $n$ is odd, we are done. Suppose $n$ is even. Let $\alpha$ be a root of $r_m$. Since $g_m=r_m\cdot s_m$, we have that  $\alpha$ is a root of $g_m$ as well. Then, 

\begin{align*}
[\F_q(\alpha):\F_q] & = \frac{[\F_{q^2}(\alpha):\F_{q^2}]\cdot [\F_{q^2}:\F_q]}{[\F_{q^2}(\alpha):\F_{q}(\alpha)]} \\
&= \frac{(n/2)t^m\cdot 2}{[\F_{q^2}(\alpha):\F_{q}(\alpha)]}\\
&= \frac{nt^m}{[\F_{q^2}(\alpha):\F_{q}(\alpha)]}
\end{align*}

Then $[\F_q(\alpha):\F_q]$ is either $nt^m$ or $nt^m/2$. In both cases, $[\F_q(\alpha):\F_q]$ is an even number, hence $\F_{q^2}\subseteq \F_q(\alpha)$. But then $\F_{q^2}(\alpha)=\F_{q}(\alpha)$ and we conclude that $[\F_q(\alpha):\F_q]=nt^m$, thus showing that a nonzero multiple of $g_m$ is irreducible over $\F_q$. 
\end{proof}

\subsection{McNay's transform}

Let $c \in \mathbb F_q$ be a non square. Assume $q$ is odd. For any $g(x) \in \mathbb F_q[x]$, McNay's transform is $$g^T_c(x)=(2x)^{\deg g} g \left( \frac{x^2+c}{2x}\right).$$

Choose $\lambda \in \mathbb F_{q^2}$ such that $\lambda^2=c$. Set
\begin{equation}\label{sigmaMcNay}
\sigma = \left[ \begin{array}{cc}
\lambda & \lambda \\
-1 & 1\\
\end{array}\right].
\end{equation}

Observing that $\sigma \cdot \infty = -\lambda$, we see that whenever $g(-\lambda) \neq 0$, we have that $$g^T_c(x)=(2\lambda)^{\deg g} g^{R_{\sigma,2}}(x).$$ Also, 
$$\eta(g;\lambda)=\frac{g(\lambda)}{g(-\lambda)}$$ differs from $g(\lambda)g(-\lambda)$ by a square in $\F_{q^2}$.  Applying Theorem \ref{quadraticit} with $t=2$ and $\sigma$ as in \eqref{sigmaMcNay}, we obtain the following result. 

\begin{theorem}\label{ThmMcNay}
Assume $q$ is odd. Let $g(x) \in \F_q[x]$ be an irreducible polynomial of degree $n$. Let $c\in \F_q$ and $\lambda \in \F_{q^2}$ be as above. Define $g_0=g$ and $g_{m}=(g_{m-1})_c^{T}$. If $n$ is even, let $r(x) \in \F_{q^2}[x]$ be an irreducible factor of $g(x)$ in $\F_{q^2}[x]$. 

If $n$ is odd, assume that $g(\lambda)g(-\lambda)$ is not a square in $\F_{q^2}$. If $n$ is even, assume that $r(\lambda)r(-\lambda)$ is  not a square in $\F_{q^2}$. Then, all polynomials in the sequence $g_m$ are irreducible in $\F_q[x]$
\end{theorem}

\noindent {\bf Example}. Assume in addition that $1-c$ is also not a square. Then, the polynomial $g_0(x)=x^2+2x+c$ is irreducible over $\F_q$. The factorization over $\F_{q^2}$ is given by $$g_0(x)=(x+1-\sqrt{1-c})(x+1+\sqrt{1-c})=:r(x)s(x).$$

Moreover,

$$r(\lambda)\cdot r(-\lambda)=2(1-c-\sqrt{1-c}).$$ 

Assume $2(1-c-\sqrt{1-c})$ is a square in $\F_{q^2}$. Since $\{1-c, \sqrt{1-c}\}$ is a basis of $\F_{q^2}$ as a vector space over $\F_q$, there are elements $a,b \in \F_q$ such that $$2(1-c-\sqrt{1-c})=(a(1-c)+b \sqrt{1-c})^2.$$ This leads to the equations
$$a^2(1-c)+b^2=2, \quad  2ab(1-c)=-2.$$

We deduce that $a^4(1-c)^3-2a^2(1-c)^2+1=0$. Since this is a quadratic equation for $a^2$, and we have that $a^2\in \F_q$, the corresponding discriminant must be a square in $\F_q$. The discriminant is 
$$-4(1-c)^2\cdot c(1-c).$$

Since $c$ and $1-c$ are not squares in $\F_q$, we have that $c(1-c)$ is a square in $\F_q$. Hence, the discriminant is a square in $\F_q$ if and only of $-1$ is a square in $\F_q$. This reasoning can be clearly reversed. We deduce that $r(\lambda)r(-\lambda)$ is a square in $\F_{q^2}$ if and only if $-1$ is a square in $\F_q$. 

In conclusion,  if $q\equiv 3 \mod 4$, we have that all polynomials $g_m$ starting with $g_0$ are irreducible. This consequence had been proven earlier by McNay. Moreover, Chapman shows that the roots of the $g_m$ are completely normal elements of $\F_{q^{2^{m+1}}}$ over $\F_q$ (see \cite{chapman}, Theorem 2).

\subsection{A transform related to Singer groups}

Let $c \in \mathbb F_q$ be such that the polynomial $x^2-x-c \in \mathbb F_q[x]$ is irreducible. Let $\theta \in \mathbb F_{q^2}$ be a root and $t\geq 2$.  Set 
$$f(x)=\frac{\theta(x + \theta^q)^t -\theta^q(x+\theta)^t}{\theta^q-\theta}, \quad  h(x)=\frac{(x+\theta)^t-(x+\theta^q)^t}{\theta^q-\theta}.$$

It is not hard to check that both polynomials belong to $\mathbb F_q[x]$. Let $Q_{c,t}(x)=f(x)/h(x)$. For a polynomial $g (x) \in \mathbb F_q[x]$ of degree $n$, set $$g^{Q_{c,t}}(x)=h(x)^n g\left(\frac{f(x)}{h(x)}\right).$$

A particular case of this transform has been studied in \cite{Panarioetal}. Setting $\sigma^\theta = \left[\begin{array}{cc}
\theta & -\theta^q \\
-1 & 1 
\end{array}\right],$ we have that $\sigma^\theta \cdot \infty=-\theta$. For any polynomial $g(x) \in \mathbb F_q[x]$ with $g(-\theta)\neq 0$, we have that 
\begin{equation}\label{Pan}
 g^{Q_{c,t}}(x)= (\theta-\theta^q)^{n(t-1)}g^{R_{\sigma^\theta,t}}(x).
\end{equation} Using this observation we deduce the following.

\begin{theorem}\label{ThmPanarioetal}
Let $g(x) \in \mathbb F_q[x]$ be an irreducible polynomial of degree $n$ such that $g(-\theta) \neq 0$. If $n$ is even, let $r(x)\in \F_{q^2}[x]$ be an irreducible factor of $g(x)$ in $\F_{q^2}[x]$. Assume that every prime factor of $t$ divides $q^2-1$.  Define $g_0=g$ and $g_{m}=g_{m-1}^{Q_{c,t}}$ for $m\geq 1$. 

If $n$ is odd (resp. if $n$ is even), assume that for all primes $\ell |t$, the element $\frac{g(-\theta^q)}{g(-\theta)}$ (resp. $\frac{r(-\theta^q)}{r(-\theta)}$) is not an $\ell$-th power in $\F_{q^2}$. Then, if $n$ is odd or if $4|nt$, all polynomials in the  sequence $g_m$ are irreducible in $\F_q[x]$. 

\end{theorem}

\begin{proof}
Since 
$$(\sigma^\theta)^{-1}\cdot \infty = 1, \quad \sigma^\theta\cdot 0=-\theta^q,$$ we have that $\eta(g;\sigma)=\frac{g(-\theta^q)}{g(-\theta)}.$ The  assertion then follows by equation \eqref{Pan} and an application of Theorem \ref{quadraticit} to $\sigma^\theta$ and  $t$.
\end{proof}

We consider $\mathbb F_q^*$ as a subgroup of ${\rm GL}_2(\mathbb F_q)$ through the embedding sending $a \in \mathbb F_q^* $  to $\left[\begin{array}{cc}
a & 0 \\
0 & a 
\end{array}\right]$. Let ${\rm PGL}_2(\mathbb F_q)= {\rm GL}_2(\mathbb F_q)/\mathbb F_q^*$ and consider the matrix
$A_c =\left[\begin{array}{cc}
0 & 1 \\
c & 1 
\end{array}\right].$ We denote by $D$ its order in ${\rm PGL}_2(\mathbb F_q)$. Since $D|q+1$, Theorem \ref{ThmPanarioetal} applies to the transform $f^{Q_{c,D}}$, thus providing a partial answer to Problem 1 in \cite{Panarioetal}, section 5.

\bibliographystyle{alpha}
\bibliography{Bib}
\end{document}